\documentclass{article}
\usepackage{amsbsy}
\usepackage{amsmath,bm,enumerate}
\usepackage{amssymb}
\usepackage{amsthm}
\usepackage{amsmath}
\usepackage{graphicx}
\usepackage{float}
\usepackage{cite}
\usepackage{subfigure}
\usepackage{color}
\usepackage{mdwlist}
\usepackage{authblk}

\allowdisplaybreaks 
\numberwithin{equation}{section}
\newtheorem{theorem}{Theorem}[section]
\newtheorem{lemma}{Lemma}[section]
\newtheorem{proposition}{Proposition}[section]
\newtheorem{corollary}{Corollary}[section]
\theoremstyle{definition}
\newtheorem{remark}{Remark}[section]
\newcommand{\R}{\mathbb R}
\newcommand{\eps}{\varepsilon}
\newcommand{\dd}{\,\mathrm d}
\newcommand{\Hrad}{H^1_{\mathrm{rad}}(\R^n)}
\newcommand{\Xrad}{C_{0,\mathrm{rad}}(\R^n)}
\newcommand{\ps}{p_{\mathrm S}(n)}
\newcommand{\eqnum}[1]{\refstepcounter{equation}\label{#1}\eqno(\theequation)}
\DeclareMathOperator{\arsinh}{arsinh}

\title{\Large \bf  Uniqueness and radial nondegeneracy of nodal states for a mixed-power scalar field equation}
\author{Zhibo Cheng\thanks{Corresponding author. E-mail: czb\_1982@126.com}
\hspace{0.25em} and \hspace{0.05em}
Sunan Wang}

\affil{\small \itshape School of Mathematics and Information Science, Henan Polytechnic University, Jiaozuo, China}

\date{}
\begin{document}

\maketitle
\begin{center}
\begin{minipage}{14cm}
{\bf Abstract}---We consider the scalar field equation
$$\begin{cases}
-\Delta u+u+\lambda |u|^{q-1}u
=\mu |u|^{p-1}u, & x\in\mathbb R^n,\\[2mm]
u(x)\to 0, & |x|\to\infty,
\end{cases}$$
 where $n\geq3$, $1<q<p<(n+2)/(n-2)$, and $\lambda,\mu>0$. We study uniqueness and radial nondegeneracy of its sign-changing radial bound states. The proof combines a comparison argument for the shooting variation across nodal intervals with uniform compactness estimates for families of bounded nodal number and continuation in a normalized parameter. We show that, for every integer $k\geq1$, fixing the centre and requiring $u(0)>0$ leaves exactly one radial bound state with $k$ nodes. Its radial linearized operator has trivial kernel in the radial space. For each fixed $k$, the corresponding normalized solutions form a global $C^1$ branch.

\end{minipage}
\end{center}

\section{Introduction}

We study radial bound states of the scalar field equation
\begin{equation}{\label{eq1.1}}
\begin{cases}
-\Delta u+u+\lambda |u|^{q-1}u
=\mu |u|^{p-1}u, & x\in\mathbb R^n,\\[2mm]
u(x)\to 0, & |x|\to\infty,
\end{cases}
\end{equation}
where $\lambda,\mu>0$, and under the assumptions
\begin{equation}{\label{eq1.2}}
n\ge3,~~
1<q<p<\ps:=\frac{n+2}{n-2}.
\end{equation}
Our focus is on radial bound states that tend to zero at infinity. We fix their centre at the origin and take $u(0)>0$. A zero of the radial profile in $(0,\infty)$ is called a node. The principal questions are whether a bound state with exactly $k$ nodes is unique for each prescribed positive integer $k$, and whether the corresponding radial linearized operator has a nontrivial kernel.

Equation \eqref{eq1.1} is also the standing-wave equation for a nonlinear Schr\"odinger equation with competing
power nonlinearities. The lower-power term is defocusing, whereas the higher-power term is focusing, and their relative strengths vary with the amplitude. Positive localized profiles correspond to ground states, while sign-changing profiles are associated with excited states. Positive bound states and sign-changing bound states thus present distinct classification problems. Work on solitary waves and their variational properties provides further context for the study of such profiles \cite{Weinstein, Strauss, Weinstein1983, Grillakis}.

The existence theory for nonlinear scalar field equations goes back to Berestycki and Lions \cite{Berestycki, Berestycki1983}; related ODE methods and existence results for solitary waves can be found in \cite{Strauss, Berestycki1981}. Under suitable growth and sign conditions, equations of this type possess both a positive ground state and infinitely many radial solutions. McLeod, Troy and Weissler \cite{McLeod} studied radial solutions with prescribed numbers of zeros, while Jones and K\"upper \cite{Jones} investigated the existence of infinitely many solutions through the structure of the radial equation. These results make the nodal number a natural way to distinguish sign-changing bound states. The existence of a $k$-node solution, however, does not determine how many solutions have that nodal number.

There is a substantial theory of uniqueness for positive solutions.
Coffman \cite{Coffman} treated the three-dimensional cubic pure-power equation, and Kwong \cite{Kwong} established uniqueness of the positive pure-power solution in the corresponding general range. Peletier and Serrin \cite{Peletier}, McLeod and Serrin \cite{McLeod1987}, McLeod \cite{McLeod1993}, and Chen and Lin \cite{Chen}, among others, developed comparison arguments and uniqueness results for more general nonlinearities; Cort\'azar, Elgueta, and Felmer \cite{Cortazar1998} also gave conditions for uniqueness of positive solutions. For the nonlinearity in \eqref{eq1.1}, the criterion of Serrin and Tang \cite{Serrin} applies to uniqueness of positive radial solutions. The questions addressed here are therefore uniqueness of sign-changing solutions with any prescribed positive nodal number and radial nondegeneracy of those solutions.

Sign-changing solutions present additional difficulties for comparison arguments. Each time a radial profile crosses a node, both the solution and its variation with respect to the initial value at the centre may change sign. Trajectories arising from different initial values may also intersect again in later nodal intervals. A comparison on the first monotonicity interval therefore cannot simply be repeated.
Cort\'azar, Garc\'ia-Huidobro and Yarur
studied uniqueness of particular nodal solutions and, subsequently, uniqueness of more general sign-changing bound states \cite{Cortazar, Cortazar2009}.
Ao, Wei and Yao \cite{Ao} obtained uniqueness and nondegeneracy when the exponent is close to the critical exponent.
Cohen, Li and Schlag \cite{Cohen} proved uniqueness of the first twenty excited states of the three-dimensional cubic equation by rigorous computer-assisted analysis. On the other hand, for certain nonlinearities, several bound states can have the same nodal number \cite{Cortazar2015}; see also \cite{Herreros} for the role of the intermediate-amplitude regime in multiplicity. Thus nodal uniqueness depends on the specific structure of the nonlinearity and does not follow from subcritical growth or existence alone.

For double-power equations, it is also essential to distinguish the signs of the power terms and whether the solutions under consideration are positive or sign-changing. Lewin and Rota Nodari \cite{Lewin} studied uniqueness, nondegeneracy, and parameter dependence of positive radial solutions for a different double-power structure. Killip, Oh, Pocovnicu, and Visan \cite{Killip} investigated solitary waves and scattering for a cubic-quintic model. For a model in which both nonlinear power terms are focusing, D\'avila, del Pino, and Guerra \cite{Davila} exhibited nonuniqueness of positive radial solutions; more recently, Song and Zou \cite{Song} studied uniqueness of positive solutions for that sign configuration under specified restrictions on the dimension and exponents. Each of these results concerns its own equation and parameter range, and none directly yields uniqueness of the $k$-node solutions of \eqref{eq1.1} for arbitrary $k$.

Tang \cite{Tang} resolved the radial nodal uniqueness problem for the pure-power equation
$$-\Delta u+u=|u|^{p-1}u~~\text{in }\R^n$$
throughout the Sobolev subcritical range when $n\ge3$: for each $k\ge1$, fixing the centre and the sign at the origin leaves exactly one $k$-node bound state. In the same paper, Tang further proposed nodal uniqueness conjectures for several classes of nonlinearities, including
$$f(s)=-s-\lambda|s|^{q-1}s+\mu|s|^{p-1}s,~~\lambda,\mu>0.$$
We prove that, in the range $n\geq3$, $1<q<p<(n+2)/(n-2)$, and $\lambda,\mu>0$, there is exactly one radial bound state for each prescribed positive nodal number. This gives an affirmative answer to Tang's conjecture in the stated double-power parameter range. Other related shooting analyses appear in \cite{Zhang}, although the regularity of the nonlinearity near zero and the tail behaviour of the bound states there differ from those in the present problem.

The first difficulty in our proof arises in the small-amplitude regime. Let
$$F(s)=\int_0^s f(t)\,dt.$$
For $s>0$,
$$sf'(s)-f(s)=\mu(p-1)s^p-\lambda(q-1)s^q,$$
and
$$sf(s)-2F(s)=\frac{\mu(p-1)}{p+1}s^{p+1}-\frac{\lambda(q-1)}{q+1}s^{q+1}.$$
Because $q<p$, both expressions are negative at sufficiently small positive amplitudes and positive at larger amplitudes. Auxiliary quotients with either expression in the denominator therefore cannot be used throughout an entire nodal interval. To carry the comparison across nodes, we retain the variational information in the large-amplitude region and use a regular bridging quantity, which involves neither denominator, to pass through the small-amplitude region. A comparison at equal absolute amplitudes and an integral estimate over the intermediate region then transfer the required positivity to the outgoing arc, allowing successive nodal intervals to be treated in turn. The same argument determines the zero structure of the shooting variation and excludes nontrivial decaying radial solutions of the linearized equation.

The second difficulty is to pass from radial nondegeneracy of individual solutions to uniqueness for all positive parameter values. A local branch issuing from a known solution does not rule out other branches; moreover, on the whole space, local convergence may be accompanied by nodes or peaks escaping to infinity. We establish uniform a priori estimates for all bound states with bounded nodal number as the parameter ranges over a compact interval. In particular, an estimate of the energy dissipated by a distant outgoing peak yields uniform control of the tails, preservation of the nodal number, and strong compactness. Combined with radial nondegeneracy, these properties show that the number of solutions with a fixed nodal number is locally constant with respect to the normalized parameter. Starting from Tang's pure-power uniqueness theorem \cite{Tang} at the zero-parameter limit, we obtain uniqueness for \eqref{eq1.1} for every $\lambda,\mu>0$ and prove that the normalized solutions form a global $C^1$ branch.

The results above establish existence, uniqueness, and radial nondegeneracy of the prescribed-node bound states, together with a global branch description, under the higher-dimensional assumptions in \eqref{eq1.2}. The two-dimensional case is treated separately in Appendix A, where the comparison theorem of \cite{Cortazar} yields an at most one result for the same configuration of coefficient signs.

The remainder of the paper is organized as follows. Section 2 states the main results. Section 3 establishes the a priori estimates and compactness needed for parameter continuation. Section 4 develops the structural inequalities, Sections 5 and 6 prove the comparison across nodal intervals and radial nondegeneracy, respectively, and Section 7 completes the continuation and solution-counting arguments.

\section{Main results}

We now state the main results. Besides uniqueness and radial nondegeneracy for each prescribed positive nodal number, we establish compactness for bound states with bounded nodal number and describe the resulting global solution branch. A normalization of the coefficients reduces the parameter dependence to a single quantity and allows the pure-power equation to serve as the endpoint of the continuation argument.

Throughout this section, assume \eqref{eq1.2}.
The centre of every radial solution is fixed at the
origin, and its value there is required to be positive.
We write
$$
H^1_{\mathrm{rad}}(\R^n)
=
\{u\in H^1(\R^n):u\text{ is radial}\}
$$
and
$$
C_{0,\mathrm{rad}}(\R^n)
=
\{u\in C_0(\R^n):u\text{ is radial}\},
$$
where \(C_0(\R^n)\) denotes the space of continuous
functions vanishing at infinity, equipped with the
uniform norm.

The transformation
$$
U=\mu^{1/(p-1)}u,
\qquad
\eps=\lambda\mu^{-(q-1)/(p-1)}
\eqnum{eq:normalization}
$$
reduces \eqref{eq1.1} to
$$
-\Delta U+U+\eps|U|^{q-1}U
=
|U|^{p-1}U
\quad\text{in }\R^n.
\eqnum{eq:normalized}
$$
We include the endpoint \(\eps=0\).
Unless otherwise stated, \(u\) denotes the normalized
profile in the results concerning
\eqref{eq:normalized}.

Our first result concerns uniqueness and radial
nondegeneracy for arbitrary positive coefficients.

\begin{theorem}\label{thm:uniqueness}
For every \(\lambda,\mu>0\) and every integer \(k\ge1\),
equation \eqref{eq1.1} admits exactly one radial
bound state satisfying \(u(0)>0\) and having exactly
\(k\) nodes.

All its nodes are simple, and
$$
|u(r)|<u(0)
\qquad\text{for every }r>0.
$$
Moreover, \(u\in H^1(\R^n)\), and \(u\) and its radial
derivative decay exponentially at infinity.

The solution is radially nondegenerate.
More precisely, if
$$
L_u
=
-\Delta+1+\lambda q|u|^{q-1}
-\mu p|u|^{p-1},
$$
then
$$
\{w\in H^1_{\mathrm{rad}}(\R^n):
L_uw=0\text{ weakly in }\R^n\}
=
\{0\}.
\eqnum{eq:radial-nondegeneracy}
$$
\end{theorem}

The nondegeneracy assertion concerns radial
perturbations.
In the full space, translation invariance gives
$$
L_u(\partial_{x_j}u)=0,
\qquad j=1,\ldots,n.
$$
No assertion about the remaining nonradial modes
is made here.

The next result gives uniform estimates and
compactness at bounded nodal number.
Here and below, \(u'\) denotes the derivative of
the radial profile.

\begin{theorem}\label{thm:compactness}
Fix \(E_*\in[0,\infty)\) and an integer \(K\ge0\).
There exist constants \(C,R,c>0\), depending only on
\(n,p,q,E_*,K\), such that every radial bound state
of \eqref{eq:normalized} with
\(u(0)>0\), \(0\le\eps\le E_*\), and at most \(K\)
nodes satisfies
\begin{equation}{\label{eq2.4a}}
\|u\|_{H^1(\R^n)}
+\|u\|_\infty+\|u'\|_\infty
\le C
\end{equation}
and
\begin{equation}{\label{eq2.5a}}
|u(r)|+|u'(r)|
\le Ce^{-c(r-R)}~~\text{for }r\ge R.
\end{equation}

Let \((\eps_j,u_j)\) be a sequence of such states,
each having exactly \(k\) nodes, where
\(0\le k\le K\) is fixed.
After passage to a subsequence, there exist
\(\bar\eps\in[0,E_*]\) and a radial bound state
\(\bar u\) of \eqref{eq:normalized} with coefficient
\(\bar\eps\), positive initial value, and exactly
\(k\) nodes, such that
$$
\eps_j\longrightarrow\bar\eps,
\qquad
u_j\longrightarrow\bar u
\quad\text{in }H^1(\R^n)\cap C_0(\R^n).
$$
In addition,
$$
u_j\longrightarrow\bar u
\quad\text{in }C^1([0,A])
\qquad\text{for every }A>0.
$$
\end{theorem}

Radial nondegeneracy and the compactness in
Theorem~\ref{thm:compactness} yield the following
global description of the normalized solutions.

\begin{corollary}\label{cor:global-branch}
For every integer \(k\ge1\), the radial bound states
of \eqref{eq:normalized} with positive initial value
and exactly \(k\) nodes form a single global
\(C^1\) map
$$\eps \longmapsto u_{k,\eps}\in C_{0,\mathrm{rad}}(\R^n),
\eqnum{eq:global-branch}$$
where $\eps \in [0,\infty)$. The derivative at \(\eps=0\) is understood as
a right derivative.

The map is also continuous into \(H^1(\R^n)\)
and into \(C^1([0,A])\) for every \(A>0\).
In particular, as \(\eps\downarrow0\),
$$
u_{k,\eps}\longrightarrow u_{k,0}
\quad\text{in }H^1(\R^n)\cap C_0(\R^n)
$$
and
$$
u_{k,\eps}\longrightarrow u_{k,0}
\quad\text{in }C^1([0,A])
\qquad\text{for every }A>0.
$$
\end{corollary}

\section{Compactness at a fixed nodal number}

In this section we prove the uniform estimates and compactness stated in Theorem \ref{thm:compactness}. We first record the energy properties and decay of individual bound states. A rescaling argument then bounds their central amplitudes when the parameter and nodal number are bounded. Finally, we exclude peaks escaping to infinity; this gives uniform tail control and allows the nodal number to be preserved under convergence.

For $\eps\ge0$, write
\begin{equation}{\label{eq2.1}}
 f_\eps(s)=-s-\eps|s|^{q-1}s+|s|^{p-1}s,~~
 F_\eps(s)=-\frac{s^2}{2}-\frac{\eps}{q+1}|s|^{q+1}+\frac{1}{p+1}|s|^{p+1}.
\end{equation}
The radial initial value problem is
$$
 u''+\frac{n-1}{r}u'+f_\eps(u)=0,~~u(0)=\alpha>0,~~u'(0)=0.
 \eqnum{eq2.2}
$$
At the origin its regular interpretation is
$$
 u'(r)=-r^{1-n}\int_0^r t^{n-1}f_\eps(u(t))\dd t,
~~u''(0)=-\frac{f_\eps(\alpha)}{n}.
 \eqnum{eq2.3}
$$
The integral equation gives local existence, uniqueness, and $C^1$ dependence on $\alpha$ and $\eps$ on bounded intervals. Since $p,q>1$, $f_\eps$ is $C^1$ even at zero. The primitive is bounded below and tends to $+\infty$ as $|s|\to\infty$. The energy identity below therefore also implies global existence for the initial value problem.

\subsection{Energy and individual tails}

We begin with estimates for individual bound states. The radial energy identity controls their amplitudes and critical points, and it also determines the behaviour of the final nodal interval. These properties provide the starting point for the uniform estimates developed later in this section.

\begin{lemma}\label{lem2.1}
The function $f_\eps$ has exactly one positive zero $b_\eps\ge1$, and $F_\eps$ has exactly one positive zero $\beta_\eps>b_\eps$. Every normalized bound state satisfies
$$
 \alpha>\beta_\eps\ge\beta_0
 =\biggl(\frac{p+1}{2}\biggr)^{1/(p-1)}>1,
 \qquad |u(r)|<\alpha\quad(r>0).
 \eqnum{eq2.4}
$$
Its energy
$$
 E(r)=\frac{u'(r)^2}{2}+F_\eps(u(r)),\qquad
 E'(r)=-\frac{n-1}{r}u'(r)^2
 \eqnum{eq2.5}
$$
is positive at every finite $r$ and tends to zero at infinity. Every nonzero critical point satisfies $|u|>\beta_\eps$. Each nodal domain other than the first contains exactly one critical point, a strict maximum on a positive domain and a strict minimum on a negative domain. The first domain starts with a strictly decreasing profile. There are finitely many nodes, and $u$ and $u'$ decay exponentially.
\end{lemma}

\begin{proof}
For $s>0$, the equation $f_\eps(s)=0$ is equivalent to $
 s^{p-q}-s^{1-q}=\eps$.
The left side is strictly increasing from $-\infty$ to $+\infty$ and is zero at $s=1$. Thus $f_\eps<0$ on $(0,b_\eps)$ and $f_\eps>0$ on $(b_\eps,\infty)$. The primitive decreases and then increases, with $F_\eps(b_\eps)<0$ and $F_\eps(s)\to\infty$. This proves the assertion about $\beta_\eps$. Since $F_\eps\le F_0$, its positive zero is at least $\beta_0$.

For a bound state, $u(r)\to0$. The energy identity and a lower bound for $F_\eps$ give a global bound for $u'$. The equation then bounds $u''$ on $[1,\infty)$. Since a convergent function with uniformly continuous derivative has derivative tending to zero, $u'(r)\to0$. To see the last assertion directly, a sequence on which $|u'|$ stays above a fixed positive number would, by uniform continuity, give intervals of a fixed length on which $u'$ has the same sign and is bounded away from zero; the resulting change in $u$ contradicts its convergence. Consequently $E(\infty)=0$.

The energy decreases strictly between any two distinct finite radii. Equality would force $u'\equiv0$ on an interval, and hence a constant solution by uniqueness; a nonzero constant cannot decay. Thus $E(r)>0$, including $E(0)=F_\eps(\alpha)>0$. If $|u(r)|\ge\alpha$ for some $r>0$, the strict increase of $F_\eps$ on $[\beta_\eps,\infty)$ would give
$$
 F_\eps(u(r))\ge F_\eps(\alpha)=E(0)>E(r),
$$
which is impossible. At a critical point, $F_\eps(u)=E>0$, so $|u|>\beta_\eps$ and $u''=-f_\eps(u)$ has the opposite sign to $u$. Two positive maxima in one positive domain would require an intervening critical point that is not a strict maximum. The negative case is identical. A node cannot be double by uniqueness for the equation at any positive radius. The behaviour at the origin follows from~\eqref{eq2.3}.

Because $u\to0$, beyond some radius all amplitudes are smaller than $\beta_\eps$. A node there would require a subsequent nonzero critical point before the solution could return to zero, contradicting the critical amplitude bound. Hence the number of nodes is finite and the final tail has a constant sign, with $|u|$ strictly decreasing. For $|s|\le1$,
$$
 -2F_\eps(s)\ge m s^2,
 \qquad m:=\frac{p-1}{p+1}>0.
 \eqnum{eq2.6}
$$
By positivity of $E$, $|u'|\ge\sqrt m\,|u|$ on the small tail. This gives exponential decay of $u$. Also $|f_\eps(u)|\le C|u|$ there, so the flux $r^{n-1}u'(r)$ has a finite limit. If that limit were nonzero, integrating $u'$ would give a nonzero $r^{2-n}$ leading term, contradicting exponential decay. Its limit is therefore zero, and
$$
 u'(r)=r^{1-n}\int_r^\infty t^{n-1}f_\eps(u(t))\dd t.
 \eqnum{eq2.7}
$$
This proves exponential decay of $u'$ as well. The equation gives the same decay for $u''$, away from the origin. In particular $u\in H^1(\R^n)$.
\end{proof}

\subsection{The oscillatory blow-up limit}

We next establish a uniform bound on the central amplitudes of bound states with bounded nodal number. If these amplitudes were unbounded, rescaling near the origin would lead to a limiting pure-power equation. The oscillation of its solution provides the contradiction needed for the amplitude bound.

\begin{lemma}\label{lem2.2}
Let $n\ge3$ and $1<p<\ps$. The solution
$$
 V''+\frac{n-1}{r}V'+|V|^{p-1}V=0,
 \qquad V(0)=1,\quad V'(0)=0
 \eqnum{eq2.8}
$$
exists globally and has infinitely many simple positive zeros.
\end{lemma}

\begin{proof}
The energy $V'^2/2+|V|^{p+1}/(p+1)$ is nonincreasing. It bounds $V$ and $V'$, so the solution is global. A double zero would imply $V\equiv0$, which is excluded by the initial condition.

Suppose there are only finitely many zeros. Reversing sign if necessary, $V>0$ on a final interval. Since
$$
 (r^{n-1}V')'=-r^{n-1}V^p<0,
$$
the derivative must eventually be negative. Otherwise the profile would be nondecreasing with a positive lower bound and the same flux identity would force $V'<0$. The positive decreasing limit is zero, again by the flux identity. Set $a=2/(p-1)$. At large $r$, integrating on $[r/2,r]$ gives
$$
 -V'(r)\ge c_n rV(r)^p.
$$
Integrating the resulting inequality for $V^{1-p}$ yields $V(r)\le Cr^{-a}$. Integrating the flux from a fixed radius $R$ also gives
$$
 |V'(r)|\le Cr^{1-n}\biggl(1+\int_R^r t^{n-1-ap}\dd t\biggr).
 \eqnum{eq2.9}
$$
Since $ap=a+2$, the right side is bounded respectively by
$$
 \begin{cases}
 C(r^{1-n}+r^{-a-1}),&n-2-a>0,\\
 Cr^{1-n}(1+\log r),&n-2-a=0,\\
 Cr^{1-n},&n-2-a<0.
 \end{cases}
$$
Define
$$
 \mathcal P(r)=r^n\biggl(\frac{V'^2}{2}
                         +\frac{|V|^{p+1}}{p+1}\biggr)
               +\frac{n-2}{2}r^{n-1}VV'.
$$
Differentiation gives
$$
 \mathcal P'(r)=\biggl(\frac{n}{p+1}-\frac{n-2}{2}\biggr)
                  r^{n-1}|V|^{p+1}.
 \eqnum{eq2.10}
$$
The coefficient is positive and $\mathcal P(0)=0$, so $\mathcal P(r)>0$ and is strictly increasing for $r>0$. Subcriticality is equivalent to $2a>n-2$. The estimates above imply
$$
 r^n|V|^{p+1}\to0,\qquad r^nV'^2\to0,
 \qquad r^{n-1}VV'\to0.
$$
For the first expression the exponent is $n-2-2a<0$. For the derivative expressions, the polynomial terms decay because $n>2$ and $2a>n-2$; in the borderline case $r^{2-n}(1+\log r)^2\to0$. Thus $\mathcal P(r)\to0$, a contradiction.
\end{proof}

\begin{proposition}\label{pro2.1}
Fix $E_*<\infty$ and $K\ge0$. The normalized bound states with $0\le\eps\le E_*$ and at most $K$ nodes satisfy
$$
 \beta_0<u(0)\le A,
 \qquad \|u\|_\infty+\|u'\|_\infty\le C
 \eqnum{eq2.11}
$$
for constants $A,C$ depending only on $n,p,q,E_*,K$.
\end{proposition}

\begin{proof}
If $\alpha_j=u_j(0)\to\infty$, rescale by
$$
 V_j(t)=\alpha_j^{-1}u_j(\alpha_j^{-(p-1)/2}t).
$$
Then $V_j(0)=1$, $V_j'(0)=0$, and
$$
 V_j''+\frac{n-1}{t}V_j'
 -\alpha_j^{1-p}V_j
 -\eps_j\alpha_j^{q-p}|V_j|^{q-1}V_j
 +|V_j|^{p-1}V_j=0.
$$
Both lower-order coefficients tend to zero. Continuous dependence in the integral equation gives $V_j\to V$ in $C^1$ on every bounded interval, where $V$ is the solution of~\eqref{eq2.8}. Choose an interval containing its first $K+1$ zeros. Their simplicity provides opposite signs on either side of each zero, so $V_j$ also has at least $K+1$ zeros for large $j$, a contradiction. Lemma~\ref{lem2.1} now bounds $\|u\|_\infty$. Finally,
$$
 \frac{u'^2}{2}\le F_\eps(\alpha)
 -\min_{|s|\le A,\ 0\le\eta\le E_*}F_\eta(s)
$$
bounds $u'$ uniformly.
\end{proof}

\subsection{Local limits and exclusion of distant peaks}
In this subsection, we prove that local limits of bound states retain the prescribed nodal structure. The main step is to exclude peaks at radii tending to infinity, which yields a uniform small-amplitude bound on the tails.

\begin{lemma}\label{lem2.3}
Every sequence in Proposition~\ref{pro2.1} has a subsequence such that $\eps_j\to\bar\eps$, $u_j(0)\to\bar\alpha\ge\beta_0$, and $u_j\to U$ in $C^1$ on bounded radial intervals. The limit is nonzero, has at most $K$ nodes, and tends monotonically to zero on its final sign interval.
\end{lemma}

\begin{proof}
Compactness of the initial data and continuous dependence give the local convergence and the limiting equation. The initial lower bound excludes the zero solution. Every zero of $U$ is simple, so $K+1$ limiting zeros would transfer to $u_j$ and contradict its nodal bound. At every finite $r$,
$$
 E_U(r)=\lim_j E_{u_j}(r)\ge0.
$$
Thus every critical point has amplitude at least $\beta_{\bar\eps}>b_{\bar\eps}$ and is a strict extremum of the appropriate sign. There is at most one such point in each sign interval. The finitely many zeros consequently imply finitely many critical points.

On the final interval, reverse sign so that $U>0$. It is eventually monotone and bounded, hence tends to a finite $\ell\ge0$. Bounds for $U'$ and $U''$ give $U'\to0$ as in Lemma~\ref{lem2.1}. If $f_{\bar\eps}(\ell)\ne0$, the equation would make $U''$ have a fixed sign bounded away from zero at large radii, contradicting $U'\to0$. Therefore $\ell$ is zero or $b_{\bar\eps}$. The second alternative would give $E_U(\infty)=F_{\bar\eps}(b_{\bar\eps})<0$, contrary to $E_U\ge0$. Hence $\ell=0$, and the final monotone profile is decreasing.
\end{proof}

\begin{lemma}\label{lem2.4}
For every fixed $0<\delta<1/2$, there is a common radius $R\ge1$ such that all states in Proposition~\ref{pro2.1} satisfy $|u(r)|<\delta$ for $r\ge R$.
\end{lemma}

\begin{proof}
In the small-amplitude region, equations~\eqref{eq2.5} and~\eqref{eq2.6} give
$$
 u'^2=2E-2F_\eps(u)\ge2E+mu^2.
 \eqnum{eq2.12}
$$
Suppose there are states $u_j$ and radii $r_j\to\infty$ with $|u_j(r_j)|\ge\delta$. Pass to the local limit $U$ in Lemma~\ref{lem2.3}. Choose a fixed $R\ge1$ such that $|U(r)|<\delta/2$ for all $r\ge R$. For large $j$, let $T_j$ be the first radius after $R$ at which $|u_j|=\delta$. Local convergence implies $T_j\to\infty$.

At $T_j$ the profile is moving away from zero. Before it can turn back towards zero, its amplitude must exceed $\beta_0>1$. Thus it reaches $2\delta$ at some first $S_j>T_j$. Equation~\eqref{eq2.12} bounds the travel time by
$$
 S_j-T_j\le D:=\frac{\log2}{\sqrt m}.
$$
Using $s=|u_j|$ on this outgoing part, energy dissipation yields
 \begin{equation}{\label{eq2.13}}
 E_{u_j}(T_j)
\geq\int_{T_j}^{S_j}\frac{n-1}{r}u_j'^2\,dr
 \geq\frac{n-1}{T_j+D}\int_\delta^{2\delta}\sqrt m\,s\,ds
 =\frac{d}{T_j+D}.
 \end{equation}
Here $d=3(n-1)\sqrt m\,\delta^2/2>0$. On $[R,T_j]$ the energy is at least this value, while $|u_j|\le\delta$. Its derivative cannot vanish there by~\eqref{eq2.12}; hence $u_j$ is monotone and can cross zero at most once. Integrating the reciprocal speed over the larger amplitude interval $[-\delta,\delta]$ gives
 \begin{equation}{\label{eq2.14}}
 T_j-R\le\int_{-\delta}^{\delta}
       \frac{\,ds}{\sqrt{2d/(T_j+D)+ms^2}}
=\frac{2}{\sqrt m}\arsinh\biggl(
       \delta\sqrt{\frac{m(T_j+D)}{2d}}\biggr).
\end{equation}
The right side grows only logarithmically in $T_j$, a contradiction.
\end{proof}

This estimate accounts for the energy dissipated by an escaping outgoing peak. Nonnegativity of the energy alone would not bound the time spent close to zero.

\subsection{Proof of the compactness theorem}
In this subsection, we complete the proof of Theorem \ref{thm:compactness}. The uniform small-amplitude bound from Lemma \ref{lem2.4} yields a common exponentially decaying tail. Together with local convergence, this gives strong compactness and shows that the number of nodes is preserved in the limit.

\begin{proof}[Proof of Theorem~\ref{thm:compactness}]
Apply Lemma~\ref{lem2.4}. Beyond its common radius $R$, there can be neither a node nor a critical point: either would require a later extremum of amplitude greater than $\beta_0$, in conflict with the uniform small tail. Thus $y=|u|$ is strictly decreasing on $[R,\infty)$, and~\eqref{eq2.12} implies
$$
 y(r)\le\delta e^{-\sqrt m(r-R)}.
 \eqnum{eq2.15}
$$
For $|s|\le\delta$ and $0\le\eps\le E_*$, $|f_\eps(s)|\le C|s|$ with a common $C$. The flux argument of Lemma~\ref{lem2.1} consequently gives~\eqref{eq2.7} for each state and the uniform estimate
$$
 |u'(r)|\le Ce^{-\sqrt m(r-R)}
 \int_0^\infty(1+s)^{n-1}e^{-\sqrt m s}\dd s
 \qquad(r\ge R).
 \eqnum{eq2.16}
$$
Here we used $(r+s)^{n-1}/r^{n-1}\le(1+s)^{n-1}$ for $r\ge1$. Together with Proposition~\ref{pro2.1}, these estimates yield the uniform
bounds in \eqref{eq2.4a} and the exponential decay estimate \eqref{eq2.5a}.

For a sequence with exactly $k$ nodes, Lemma~\ref{lem2.3} gives a local limit $U$. The uniform exponential estimates pass to the limit, so $U$ is a bound state with finite energy. All nodes lie in a common bounded interval. None can approach the origin because $U(0)\ge\beta_0>0$. Every limiting zero is simple; in a sufficiently small neighbourhood its derivative has one sign, and $C^1$ convergence gives exactly one zero of $u_j$ there. On the rest of the compact interval $|U|$ is bounded away from zero. Finally, the uniform tail is node-free. It follows that $U$ has exactly $k$ nodes.

The local convergence and uniform exponential bounds imply convergence in $C_0(\R^n)$. They also provide an integrable bound for $(|u_j|^2+|u_j'|^2)r^{n-1}$ and for the corresponding difference from $U$. Dominated convergence on the tail, together with local $C^1$ convergence, gives strong convergence in $H^1(\R^n)$.
\end{proof}

\section{Structure and auxiliary identities}\label{sec3}
We now turn to the analysis of a fixed nodal bound state. This section establishes the structural inequalities and differential identities needed to track its shooting variation across successive nodal intervals.

For Sections~\ref{sec3}-\ref{sec5}, fix one normalized $k$-node bound state, with $k\ge1$, and suppress the subscript $\eps$ on $f$ and $F$. Let $b$ and $\beta$ be their positive zeros. Write the nodes as $z_1<\cdots<z_k$ and the subsequent critical points as $c_1<\cdots<c_k$, with $c_0=0$. Thus each finite phase $(c_{i-1},c_i)$ is strictly monotone and contains exactly one node $z_i$. The final phase $(c_k,\infty)$ is monotone towards zero. Energy dissipation gives
$$
 |u(c_{i-1})|>|u(c_i)|>\beta\qquad(1\le i\le k).
 \eqnum{eq3.1}
$$

\subsection{Positive Pohozaev quantities}
In this subsection, we establish positive quantities associated with the radial equation. Their differential identities yield monotonicity along the profile and provide sign information that will be used in the subsequent comparison of the shooting variation.

Define
$$
 \begin{aligned}
 P&=r^n(u'^2+2F(u))+(n-2)r^{n-1}uu',\\
 P_1&=r^n(u'^2+uf(u))+(n-2)r^{n-1}uu',\\
 P_2&=r^n\biggl(u'^2+\frac{n-2}{n}uf(u)\biggr)
                         +(n-2)r^{n-1}uu'.
 \end{aligned}
 \eqnum{eq3.2}
$$
Direct differentiation using equation~\eqref{eq2.2} yields
$$
 \begin{aligned}
 P'&=r^{n-1}\{2nF(u)-(n-2)uf(u)\},\\
 nP_2'&=-r^nu'\{(n+2)f(u)-(n-2)uf'(u)\},~~~
 (P/r^n)'&=-nP_2/r^{n+1}.
 \end{aligned}
 \eqnum{eq3.3}
$$

\begin{lemma}\label{lem3.1}
For every $r>0$, $P(r)>0$ and $P_2(r)>0$. Hence $P(r)/r^n$ is strictly decreasing. On intervals where $|u|\ge b$, one also has $P_1>0$ and
$$
 \omega=-\frac{ru'}{u},\qquad
 \omega'=\frac{P_1}{r^{n-1}u^2}>0.
 \eqnum{eq3.4}
$$
The assertion about $\omega$ is used only where $|u|\ge b$.
\end{lemma}

\begin{proof}
For $t\in\{p,q\}$ set
$$
 C_t=\frac{2n}{t+1}-(n-2)>0,
 \qquad D_t=n+2-(n-2)t>0.
$$
For $s>0$ the expressions in~\eqref{eq3.3} are
$$
 \begin{aligned}
 2nF(s)-(n-2)sf(s)&=-2s^2-\eps C_qs^{q+1}+C_ps^{p+1},\\
 (n+2)f(s)-(n-2)sf'(s)&=-4s-\eps D_qs^q+D_ps^p.
 \end{aligned}
 \eqnum{eq3.5}
$$
Divide the first by $s^{q+1}$ and the second by $s^q$. Both resulting functions are strictly increasing from $-\infty$ to $+\infty$. Each expression therefore has a unique positive zero and changes sign from negative to positive.

On an incoming arc, where $|u|$ decreases from a critical amplitude to zero, $P'$ can change sign only from positive to negative. A minimum of $P$ on that arc is consequently attained at an endpoint. At a node $P(z_i)=z_i^nu'(z_i)^2>0$, and at a nonzero critical point $P(c_i)=2c_i^nE(c_i)>0$. For the first arc $P(0)=0$, and for the final arc $P(\infty)=0$ by exponential decay. The same derivative pattern, together with these endpoint values, gives strict positivity in the interior also in these two cases. On an outgoing arc, $uu'>0$ and
$$
 P=2r^nE+(n-2)r^{n-1}uu'>0.
$$
This proves $P>0$ everywhere.

We next prove that $P_2(r)>0$ for every $r>0$. On an incoming
arc, $u u'<0$. By \eqref{eq3.3} and the sign structure of the second
expression in \eqref{eq3.5}, $P_2'$ can change sign only from positive
to negative. Thus $P_2$ cannot have an interior minimum on such
an arc. At a node $z_i$,
$$
P_2(z_i)=z_i^n u'(z_i)^2>0,
$$
while at a nonzero critical point $c_i$,
$$
P_2(c_i)=\frac{n-2}{n}c_i^n u(c_i)f(u(c_i))>0.
$$
The remaining endpoint values are $P_2(0)=0$ and
$P_2(\infty)=0$, the latter following from exponential decay.
These endpoint values and the derivative pattern give $P_2>0$
in the interior of every incoming arc.

On an outgoing arc, $u u'>0$, so the last term in the definition
of $P_2$ is positive. Put $s=|u|$. If $s\geq b$, then
$u f(u)\geq0$, and the definition of $P_2$ immediately gives
$P_2>0$. If $0<s<b$, the positivity of the energy yields
$u'^2>-2F(u)$. Consequently,
\begin{equation*}
\frac{P_2}{r^n}> -2F(u)+\frac{n-2}{n}u f(u)
=\frac{s^2}{n}
  \left(2+\varepsilon C_qs^{q-1}-C_ps^{p-1}\right)
>\frac{s^2}{n}
  \left(2-C_p+(C_q-C_p)\varepsilon s^{q-1}\right)>0.
\end{equation*}
Here we used $s^{p-1}<1+\varepsilon s^{q-1}$ for $s<b$,
$C_q>C_p$, and
$$
2-C_p=\frac{n(p-1)}{p+1}>0.$$
Positivity at the node and at the critical point follows from
the endpoint formulas above. Hence \(P_2(r)>0\) for all \(r>0\).

Set
$$
 h(s)=sf(s)-2F(s),\qquad d(s)=sf'(s)-f(s)=h'(s).
 \eqnum{eq3.6}
$$
For $s\ge b$, put $A=s^{p-1}$ and $B=\eps s^{q-1}$; then $A\ge1+B$ and
$$
 h(s)=s^2\biggl(\frac{p-1}{p+1}A
                  -\frac{q-1}{q+1}B\biggr)>0.
$$
Since $h$ is even, $P_1=P+r^nh(u)>0$ when $|u|\ge b$. Differentiating $\omega$ gives~\eqref{eq3.4}.
\end{proof}

\subsection{Large-amplitude algebra}
In this subsection, we establish the sign and monotonicity properties of the auxiliary functions used when $|u|\geq b$. These properties make the comparison quantities well defined in the large-amplitude region and provide the inequalities needed in later nodal intervals.

Define the even quotients
$$
 g_1(s)=\frac{2f(s)}{d(s)},\qquad
 g(s)=\frac{2F(s)}{h(s)}\qquad(|s|\ge b).
 \eqnum{eq3.7}
$$
Both quotients are considered only for $|s|\geq b$, where their
denominators do not vanish. The quotient $g_1$ is used on the
first incoming arc with $u\geq b$, whereas $g$ is also used on
subsequent arcs with $|u|\geq b$. Since $h>0$ on this region,
the sign of $g$ is determined by that of $F$: namely,
$g(s)<0$ for $b\leq|s|<\beta$, $g(s)=0$ for $|s|=\beta$,
and $g(s)>0$ for $|s|>\beta$.
Neither quotient is used across a zero of its denominator.

\begin{lemma}\label{lem3.2}
For $s\ge b$, $d(s)>0$, $f'(s)>0$, $f''(s)>0$, $g_1'(s)>0$, and $g'(s)>0$. Moreover,
\begin{equation}{\label{eq3.8}}
g_1(s)-g(s)=\frac{g'(s)h(s)}{d(s)}>0,~~g(\beta)=0,~~
0<g(s)<\frac{2}{p-1}~(s>\beta).
\end{equation}
In particular $sg'(s)>0$ for $|s|\ge b$.
\end{lemma}

\begin{proof}
Write $a=p-1$, $\ell=q-1$, $A=s^a$, $B=\eps s^\ell$, so $a>\ell>0$ and $A\ge1+B$. Then
$$
 \begin{aligned}
 d(s)/s&=aA-\ell B>0,~~~
 f'(s)=-1+(a+1)A-(\ell+1)B>0,\\
 sf''(s)&=a(a+1)A-\ell(\ell+1)B>0.
 \end{aligned}
$$
The last derivative is needed only away from zero. Differentiating $g_1$ gives
$$
 \frac{s}{2}g_1'(s)=
 \frac{a^2A-\ell^2B+(a-\ell)^2AB}{(aA-\ell B)^2}>0.
 \eqnum{eq3.9}
$$
If $\mathcal H=aA/(a+2)-\ell B/(\ell+2)>0$, differentiation of $g$ gives
$$
 \frac{s}{2}g'(s)=\frac{1}{\mathcal H^2}
 \biggl\{\frac{a^2A}{2(a+2)}-\frac{\ell^2B}{2(\ell+2)}
       +\frac{(a-\ell)^2AB}{(a+2)(\ell+2)}\biggr\}>0.
 \eqnum{eq3.10}
$$
Indeed $t^2/(t+2)$ is strictly increasing for $t>0$, and $A\ge1+B$. The identity for $g_1-g$ follows from $h'=d$ and $F'=f$. Finally,
$$
 h(s)-(p-1)F(s)
 =\frac{p-1}{2}s^2+\frac{\eps(p-q)}{q+1}s^{q+1}>0.
$$
For $s>\beta$ this gives the upper bound in~\eqref{eq3.8}; the lower bound follows from $F>0$ and $h>0$. Evenness supplies the assertion on negative amplitudes.
\end{proof}

\subsection{Variation quantities and the regular bridge}
We now introduce the shooting variation and the quantities used to track it along the radial profile. The identities derived below govern the comparison in the large-amplitude region, while a regular quantity carries the relevant information through the small-amplitude region, where the auxiliary quotients of the preceding subsection are unavailable.

Let $v=\partial_\alpha u$ be the regular shooting variation along the fixed state:
$$
 v''+\frac{n-1}{r}v'+f'(u)v=0,
 \qquad v(0)=1,\quad v'(0)=0.
 \eqnum{eq3.11}
$$
Define
$$
 \begin{aligned}
 Q&=r^n(u'v'+f(u)v)+(n-2)r^{n-1}u'v,\\
 M&=r^{n-1}(u'v-uv'),~~ Q_j=Q+jr^{n-1}u'v,\\
 \rho&=r^{n-1}\{f'(u)u'v-f(u)v'\},\\
 T_1&=Q-g_1(u)M,~~T=Q-g(u)M.
 \end{aligned}
 \eqnum{eq3.12}
$$
Here $j$ is a numerical subscript; in particular, the subscript
of $Q_n$ is the spatial dimension and is unrelated to the
exponent $q$. The quantities $T_1$ and $T$ are defined wherever
$|u|\geq b$. We use $T_1$ only on the first incoming arc with
$u\geq b$, and $T$ on the portions of the trajectory satisfying
$|u|\geq b$. In particular, no sign condition on $g(u)$ is
implicit in the definition of $T$.

The differential identities used in the proof are
$$
 \begin{aligned}
 Q'&=2r^{n-1}f(u)v,& M'&=r^{n-1}d(u)v,\\
 Q_j'&=r^{n-1}\{ju'v'+(2-j)f(u)v\},&
 \rho'&=r^{n-1}f''(u)u'^2v,\\
 T_1'&=-g_1'(u)u'M,&
 T'&=g'(u)\{r^{n-1}h(u)v-u'M\}.
 \end{aligned}
 \eqnum{eq3.13}
$$
The identity for $\rho$ is used only for $|u|\ge b$, where $f''$ exists and has the required sign.

On a monotone arc, where $u'\ne0$, set
\begin{equation}{\label{eq3.14}}
 \Psi=\frac{Q_n}{ru'^2},~~
 F_\eta(s)=F(s)-\frac{\eta}{2}h(s),~~
 B_\eta=Q-\eta M-2F_\eta(u)\frac{r^{n-1}v}{u'}.
\end{equation}
Here $\eta$ is a constant. Unlike $T$, the bridge $B_\eta$ is defined throughout such an arc, including its node. Its derivative and two quotient identities are
\begin{equation}{\label{eq3.15}}
 B_\eta'=-2F_\eta(u)\Psi,~~
 (v/u')'=\frac{Q_1}{r^nu'^2},~~
 \bigl(Q/(r^{n-1}u')\bigr)'=\frac{f(u)Q_2}{r^{n-1}u'^2}.
\end{equation}
For later use, where also $u\ne0$ and $h(u)\ne0$, put
$$
 Z=M-\frac{r^{n-1}h(u)v}{u'}.
$$
The connecting identities are
$$
 Q-P\frac{v}{u}=\omega Z,
 \qquad T'=-g'(u)u'Z,
 \qquad T-B_0=\frac{g(u)}{g'(u)u'}T'.
 \eqnum{eq3.16}
$$

For completeness, the bridge differentiation can be checked without differentiating a singular quotient. If $W=r^{n-1}v/u'$, then substitution of the two differential equations gives $W'=Q_n/(ru'^2)=\Psi$. For any differentiable function $H$, the derivative of $2H(u)W$ is $2H'(u)r^{n-1}v+2H(u)\Psi$. Taking $H=F_\eta$, and using $2F_\eta'=2f-\eta d$, proves the first identity in~\eqref{eq3.15}. The other identities follow by the quotient rule. For~\eqref{eq3.16}, expand $Q-Pv/u$ and factor $-ru'/u$; the factor remaining is precisely $Z$. These computations require only $f\in C^1$, except for the explicitly restricted identity for $\rho$.

\section{Comparison across one node}\label{sec4}

This section develops the comparison across a single nodal phase. We first establish the required signs on the initial phase and then prove comparison and positivity estimates that can be applied after each subsequent critical point.

\subsection{The first phase}
In this subsection we begin with the phase from the origin to the first nonzero critical point. The regular initial conditions allow us to locate the zero of the shooting variation and establish the sign properties needed to initiate the comparison across later phases.

\begin{lemma}\label{lem4.1}
On $[0,c_1]$ the variation $v$ has exactly one zero $\tau$, and $u(\tau)>b$. It satisfies
$$
 v>0\ \text{on }(0,\tau),\qquad
 v'<0\ \text{on }(0,\tau),\qquad
 v<0\ \text{on }(\tau,c_1].
$$
One has $Q_1,Q_2>0$ on $(0,c_1]$ and $Q(c_1),M(c_1)>0$. Let $r_\beta$ and $\bar r_\beta$ be the incoming and outgoing radii at which $u=\beta$ and $u=-\beta$, respectively. Then
$$
 \begin{gathered}
 Q>0\ \text{on }(0,r_\beta]\cup[\bar r_\beta,c_1],
 \qquad Q(\bar r_\beta)>Q(r_\beta),\\
 M>0\ \text{on }(0,\tau],\qquad
 T>0\ \text{on }(0,r_\beta].
 \end{gathered}
 \eqnum{eq4.1}
$$
\end{lemma}

\begin{proof}
Let $r_b$ be the radius with $u(r_b)=b$. If $v$ had no zero before $r_b$, it would be positive there. Since $\rho(0)=0$ and $\rho'>0$ before $r_b$, one would have $\rho(r_b)>0$. But $f(b)=0$ gives
$$
 \rho(r_b)=r_b^{n-1}f'(b)u'(r_b)v(r_b)\le0,
$$
a contradiction, including the possibility $v(r_b)=0$. Thus the first zero $\tau$ lies in $(0,r_b)$ and $v'(\tau)<0$. On $(0,\tau)$, the variation flux satisfies $(r^{n-1}v')'=-r^{n-1}f'(u)v<0$, whence $v'<0$. Equations~\eqref{eq3.13} give $Q,M>0$ and $T_1'>0$, so $T_1>0$ there.

We next show that $M$ cannot vanish on $[\tau,r_b]$. If its first zero occurred while $v<0$, then $T_1$ would still be increasing up to that point and $Q=T_1>0$ there. On the other hand $M=0$ implies $Q=P_1v/u$, which is negative by Lemma~\ref{lem3.1}. A second zero of $v$ cannot occur first: it would have $v'>0$, hence $M=-r^{n-1}uv'<0$, so $M$ would already have vanished. Thus $M>0$ and $v<0$ up to $r_b$. Since $g_1(b)=0$, $Q(r_b)=T_1(r_b)>0$. The derivative of $Q$ is positive before $\tau$ and negative between $\tau$ and $r_b$, proving $Q>0$ on $(0,r_b]$.

On $(r_b,z_1)$, $f(u)<0$ and $v<0$ imply $Q'>0$. A second zero of $v$ in this interval would give $Q=r^nu'v'<0$, contradicting that positivity. Therefore $v(z_1)<0$ and $Q(z_1)>0$.

On the outgoing interval $[z_1,\bar r_\beta]$, use $B_0$ to maintain $Q_1>0$. At $z_1$ one already has $Q_1=Q+r^{n-1}u'v>0$. Before a first zero of $Q_1$, $u'v>0$ and $Q_n\ge Q_1>0$, so $B_0'>0$ because $F(u)<0$. At such a first zero,
$$
 B_0=-r^{n-1}u'v-2F(u)\frac{r^{n-1}v}{u'}
     =-2r^{n-1}\frac{v}{u'}E<0,
 \eqnum{eq4.2}
$$
contrary to $B_0(z_1)=Q(z_1)>0$. If $v$ returned to zero first, its positive derivative there would make $Q_1=r^nu'v'<0$, forcing an earlier zero of $Q_1$. Hence $v<0$ and $Q_1>0$ continue to $\bar r_\beta$. At that endpoint $Q=B_0>Q(z_1)>0$. On $[\bar r_\beta,c_1]$, $f(u)<0$ and $v<0$ give $Q'>0$. A further zero of $v$ would make $Q<0$ in the interior, or $Q(c_1)=0$ if it occurred at the endpoint, both impossible. In particular $Q(c_1)>0$.

To check $Q_1,Q_2$ throughout the first phase, observe first that on $(0,\tau)$,
$$
 Q_2'=2r^{n-1}u'v'>0,
 \qquad Q_1=Q_2-r^{n-1}u'v>Q_2>0.
$$
On $(\tau,z_1]$ use $Q>0$ and $u'v>0$; on $[z_1,\bar r_\beta]$ use $Q_1>0$ and $Q_2>Q_1$; and on $[\bar r_\beta,c_1]$ again use $Q>0$ and $u'v\ge0$.

If $\tau\leq r_\beta$, then $Q_n>0$ between the two
$\beta$-crossings. Hence
$$B_0'=-2F(u)\frac{Q_n}{r(u')^2}\geq0,$$
with strict inequality away from the node. Since
$B_0=Q$ at both crossings, integration gives
$Q(\bar r_\beta)>Q(r_\beta)$.
If $\tau>r_\beta$, then $Q(\tau)>Q(r_\beta)$.
Moreover, $B_0'\geq0$ on $(\tau,\bar r_\beta)$,
with strict inequality away from the node.
Using $B_0(\tau)=Q(\tau)$ and
$B_0(\bar r_\beta)=Q(\bar r_\beta)$ gives the same
strict endpoint comparison. Since $\beta>b$ and $u$ decreases on the incoming arc,
$r_\beta<r_b$. The positivity of $M$ established up to $r_b$
therefore holds throughout $(0,r_\beta]$. On this interval
$T_1>0$, and Lemma \ref{lem3.2} gives
$$T=T_1+(g_1-g)M>0.$$
This proves all signs in~\eqref{eq4.1}.

Finally let $\bar r_b$ satisfy $u(\bar r_b)=-b$. Then
$$
 \rho(\bar r_b)=\bar r_b^{n-1}f'(b)u'(\bar r_b)v(\bar r_b)>0.
$$
On $(\bar r_b,c_1)$, $f''(u)<0$ and $v<0$, so $\rho'>0$. At $c_1$,
$$
 0<\rho(c_1)=-c_1^{n-1}f(u(c_1))v'(c_1).
$$
Since $f(u(c_1))<0$, this implies $v'(c_1)>0$ and $M(c_1)=-c_1^{n-1}u(c_1)v'(c_1)>0$.
\end{proof}

\subsection{Equal-amplitude comparison}
We next compare the incoming and outgoing arcs at the same absolute amplitude. Under the sign assumptions below, this comparison shows that the normalized quantity $\phi$ is strictly larger on the outgoing arc, a fact used to control contributions from the two large-amplitude regions.

Consider a finite phase $[c_-,c_+]$. Reflecting both functions preserves the auxiliary quantities. We may therefore arrange that $u$ decreases from a positive maximum to a negative minimum. Let $\tau$ be the only zero of $v$ in the phase. In the next three lemmas we assume
$$
 \begin{gathered}
 Q_1,Q_2>0\text{ in the phase interior},\quad
 Q>0\text{ on }(c_-,r_\beta]\cup[\bar r_\beta,c_+],\\
 Q(\bar r_\beta)>Q(r_\beta),\quad
 M>0\text{ on }(c_-,\tau],\quad u(\tau)>b,\\
 v>0, v'<0\text{ on }(c_-,\tau),\quad
 v<0\text{ on }(\tau,c_+],\quad
 T>0\text{ on }(c_-,r_\beta].
 \end{gathered}
 \eqnum{eq4.3}
$$
At $c_-=0$ the auxiliary quantities vanish; at any other left endpoint we will have $Q,Q_1,Q_2,M,T>0$. Lemma~\ref{lem4.1} proves the assumptions for the first phase. Section~\ref{sec5} will establish them independently for each later phase before invoking the comparison lemmas.

For $\sigma\in[\beta,|u(c_+)|)$, let $r_-(\sigma)$ and $r_+(\sigma)$ be the unique radii where $u=\sigma$ and $u=-\sigma$. Put
$$
 \phi=\frac{Q}{r^{n-1}|u'|},\qquad
 \phi_2=\frac{Q_2}{r^{n-1}|u'|^3}.
$$

\begin{lemma}\label{lem4.2}
Under~\eqref{eq4.3},
$$
 \phi(r_+(\sigma))>\phi(r_-(\sigma))
 \qquad(\beta\le\sigma<|u(c_+)|).
 \eqnum{eq4.4}
$$
\end{lemma}

\begin{proof}
Strict energy dissipation gives $|u'(r_+)|<|u'(r_-)|$. Also,
$$
 (r^{2n-2}E)'=2(n-1)r^{2n-3}F(u)<0
 \quad\text{on }(r_\beta,\bar r_\beta)\setminus\{z\},
$$
where $z$ is the node. Since $F=0$ at the endpoints, this implies
$$
 \bar r_\beta^{n-1}|u'(\bar r_\beta)|
 <r_\beta^{n-1}|u'(r_\beta)|.
$$
The assumed comparison of $Q$ therefore gives~\eqref{eq4.4} at $\sigma=\beta$.

At any matched pair for which $\phi(r_+)\ge\phi(r_-)$, we claim that $\phi_2(r_+)>\phi_2(r_-)$. If $r_-\le\tau$, then $Q_2(r_-)\le Q(r_-)$ and $Q_2(r_+)>Q(r_+)$. Combining this with the positive $\phi$ and the strict speed comparison proves the claim. If $r_->\tau$, both $v$ and $u'$ are negative. By~\eqref{eq3.15},
$$
 (v/u')'>0,
 \qquad \frac{|v(r_+)|}{|u'(r_+)|}
       >\frac{|v(r_-)|}{|u'(r_-)|}.
$$
Use the identity
$$
 \phi_2=\frac{\phi}{u'^2}+\frac{2|v|}{u'^2}\quad(r>\tau).
$$
The first term is larger at $r_+$ by the $\phi$ and speed comparisons. The second is larger because both $|v|/|u'|$ and $1/|u'|$ are larger there. This proves the claim in the second case.

On the incoming arc, $\sigma=u(r_-(\sigma))$, whereas on the
outgoing arc, $\sigma=-u(r_+(\sigma))$. Since $u'<0$ on both arcs,
the last identity in \eqref{eq3.15}, together with the oddness of $f$, gives
$$\frac{d}{d\sigma}\phi(r_\pm(\sigma))=f(\sigma)\phi_2(r_\pm(\sigma)).$$
Subtracting these identities and integrating from $\beta$ to
$\sigma$ yields
\begin{equation}{\label{eq4.5}}
 \phi(r_+(\sigma))-\phi(r_-(\sigma))=\phi(\bar r_\beta)-\phi(r_\beta)
+\int_\beta^\sigma f(s) \{\phi_2(r_+(s))-\phi_2(r_-(s))\}\, ds.
\end{equation}
As long as the difference on the left is nonnegative, the integrand is positive. Since its initial value is positive and $f(s)>0$ for $s\ge\beta$, the difference cannot have a first zero. This proves~\eqref{eq4.4} throughout the stated range.
\end{proof}

\subsection{The middle-interval integral}

The following estimate controls a weighted integral of $Q_n$
without requiring its pointwise positivity before $\tau$.
It is related to the integral estimate of \cite[Lemma 4.3]{Zhang}. We give a self-contained proof for the present
mixed-power nonlinearity, using the sign information available
on the interval $u>b$.

\begin{lemma}\label{lem4.3}
Suppose~\eqref{eq4.3} holds and $\tau>r_\beta$. For every constant $a\ge0$,
$$
 \int_{r_\beta}^{\tau}
       \{a h(u)-2F(u)\}\frac{Q_n}{ru'^2}\dd r>0.
 \eqnum{eq4.6}
$$
\end{lemma}

\begin{proof}
On $[r_\beta,\tau)$, $u>b$, $u'<0$, $v>0$, and $v'<0$. Define
$$
 A=-\frac{r^{n-1}v}{u'}>0,\qquad
 S=-\frac{rv'}{v}-\frac{rf(u)}{u'},\qquad
 \Psi=\frac{Q_n}{ru'^2}.
$$
The function $A$ extends continuously to $\tau$, with $A(\tau)=0$; no value of $S$ at $\tau$ is needed. By~\eqref{eq3.15}, $A'=-\Psi$. Also
$$
 Q_2=r^{n-1}u'v(n-S)>0.
$$
Since $u'v<0$, this yields $S>n$.

At $r_\beta$, positivity of $P$ and $F(\beta)=0$ give
$$
 P(r_\beta)=r_\beta^{n-1}|u(r_\beta)u'(r_\beta)|
                 \{\omega(r_\beta)-(n-2)\}>0.
$$
Thus $\omega(r_\beta)>n-2$. Since $u\ge u(\tau)>b$, Lemma~\ref{lem3.1} gives $\omega'>0$ throughout this interval. Consequently
\begin{equation*}
\frac{A'}{A}=\frac{2(n-1)-S}{r},~~~
\frac{(Au)'}{Au}=\frac{2(n-1)-S-\omega}{r}<0.
\end{equation*}
Here the first identity follows by differentiating $A$ and replacing $u''$ with the radial equation. Therefore
$$
 A(r)<A(r_\beta)\frac{\beta}{u(r)}
 \qquad(r_\beta<r<\tau).
 \eqnum{eq4.7}
$$

Split the integral in~\eqref{eq4.6} into $I_0+aI_h$. By the bridge identity,
\begin{equation*}
 I_0=\int_{r_\beta}^{\tau}(-2F(u))\Psi\,dr=B_0(\tau)-B_0(r_\beta)=Q(\tau)-Q(r_\beta)>0.
\end{equation*}
The endpoint equalities use $v(\tau)=0$ and $F(\beta)=0$; the strict inequality uses $Q'=2r^{n-1}f(u)v>0$ in the interior. Since $h'=d>0$ for $u>b$, integration by parts and~\eqref{eq4.7} give
$$
 \begin{aligned}
 I_h&=-\int_{r_\beta}^{\tau}h(u)A'\dd r\\
 &=h(\beta)A(r_\beta)
       -\int_{r_\beta}^{\tau}d(u)(-u')A\dd r\\
 &>A(r_\beta)\biggl\{h(\beta)
                  -\beta\int_{u(\tau)}^\beta\frac{d(s)}{s}\dd s\biggr\}\\
 &=A(r_\beta)\beta\biggl\{f(u(\tau))
                   +\int_{u(\tau)}^\beta\frac{f(s)}{s}\dd s\biggr\}>0.
 \end{aligned}
 \eqnum{eq4.8}
$$
The last equality follows from
$h(\beta)=\beta f(\beta)$ and
$d(s)/s=f'(s)-f(s)/s$.
Since $d(u)(-u')>0$ on $(r_\beta,\tau)$, replacing $A$
by its strict upper bound in~(5.7) increases the integral
being subtracted and therefore yields a strict lower bound
for $I_h$. The final expression is positive because
$u(\tau)>b$ and $f(s)>0$ for $s>b$.
Thus $I_0>0$ and $I_h>0$, which proves~\eqref{eq4.6} for every $a\geq0$.
\end{proof}

\subsection{Restoring positivity on the outgoing arc}
We now use the equal-amplitude comparison and the middle-interval estimate to recover positivity of $T$ on the outgoing arc. The argument controls its possible interior minimum and completes the comparison across the phase.

\begin{lemma}\label{lem4.4}
Under~\eqref{eq4.3}, $T>0$ on $[\bar r_\beta,c_+]$. In particular $T(c_+)>0$.
\end{lemma}

\begin{proof}
On the outgoing large-amplitude interval, $uv>0$, $u'v>0$ in the interior, and $Q_n>0$. At its endpoints,
$$
 T(\bar r_\beta)=Q(\bar r_\beta)>0,
 \qquad T'(c_+)=g'(u(c_+))c_+^{n-1}h(u(c_+))v(c_+)>0.
$$
Let $t$ be an interior critical point of $T$. Then $T'(t)=0$
and $g(u(t))>0$. Differentiating the last identity in \eqref{eq3.16}
at $t$ gives
\[
-B_0'(t)
 =\frac{g(u(t))}{g'(u(t))u'(t)}\,T''(t).
\]
Using $B_0'=-2F(u)\Psi$ and $g=2F/h$, we obtain
$$
 T''(t)=h(u(t))g'(u(t))u'(t)\Psi(t)>0.
 \eqnum{eq4.9}
$$
Indeed $h>0$, $g'(u)u'>0$ and $\Psi>0$. Every critical point is therefore a strict minimum, so there is at most one. If there is none, $T'>0$ throughout, and the assertion follows from its left endpoint. Otherwise denote the unique minimum by $t_+$; then $T'<0$ before $t_+$ and $T'>0$ after it.

Set $a_*=|u(t_+)|>\beta$. By~\eqref{eq3.1} there is a unique incoming radius $t_-\in(c_-,r_\beta)$ with $u(t_-)=a_*$. Let
$$
 \eta=g(a_*)\in\biggl(0,\frac{2}{p-1}\biggr),\qquad
 H_*(s)=-2F_\eta(s)=-2F(s)+\eta h(s).
$$
At both $t_-$ and $t_+$, $F_\eta(u)=0$, so $B_\eta=T$ there. Hence
$$
 T(t_+)-T(t_-)=\int_{t_-}^{t_+}H_*(u)\Psi\dd r.
 \eqnum{eq4.10}
$$
Although this interval crosses possible zeros of $h$, neither the weight nor the bridge has $h$ in a denominator.

We first check the weight. Write
$$
 C_q=\frac{2\eps}{q+1}\biggl(1-\frac{\eta(q-1)}{2}\biggr)\ge0,
 \qquad C_p=\frac{2}{p+1}\biggl(1-\frac{\eta(p-1)}{2}\biggr)>0.
$$
Since $H_*(a_*)=0$, $C_pa_*^{p-1}=1+C_qa_*^{q-1}$. It follows that
\begin{equation}{\label{eq4.11}}
H_*(s)=s^2\biggl(1-\biggl|\frac{s}{a_*}\biggr|^{p-1}\biggr)+C_q|s|^{q+1}
       \biggl(1-\biggl|\frac{s}{a_*}\biggr|^{p-q}\biggr)>0,~~(0<|s|<a_*).
\end{equation}
Next pair the incoming and outgoing large amplitudes. For $\sigma\in[\beta,a_*)$, abbreviate $r_-=r_-(\sigma)$ and $r_+=r_+(\sigma)$. If $r_-\ge\tau$, then $Q_n(r_-),Q_n(r_+)>0$: after $\tau$ this follows from $Q_n=Q_1+(n-1)r^{n-1}u'v$, and at $\tau$ from $Q_n=Q=r^nu'v'>0$. Therefore
$$
 \frac{\Psi(r_+)}{|u'(r_+)|}
 +\frac{\Psi(r_-)}{|u'(r_-)|}>0.
 \eqnum{eq4.12}
$$
If $r_-<\tau$, the incoming signs give $M>0$, $\omega>0$ and $r^{n-1}h(u)v/u'<0$. The first identity in~\eqref{eq3.16} implies
$$
 |v(r_-)|<\frac{\sigma Q(r_-)}{P(r_-)}.
$$
On the outgoing part before $t_+$, $T'<0$ and $g'(u)u'>0$, so $Z>0$ and $\omega<0$. The same identity gives
$$
 |v(r_+)|>\frac{\sigma Q(r_+)}{P(r_+)}.
$$
Also $Q_2(r_-)>0$ and $Q(r_+)>0$ imply
\begin{equation*}
 -Q_n(r_-)<(n-2)r_-^{n-1}|u'(r_-)v(r_-)|,~~~
 Q_n(r_+)>nr_+^{n-1}|u'(r_+)v(r_+)|.
\end{equation*}
Use Lemmas~\ref{lem3.1} and~\ref{lem4.2}, together with strict energy dissipation, to compare the positive factors in
$$
 \begin{aligned}
 \frac{nr_+^{n-2}Q(r_+)}{P(r_+)u'(r_+)^2}
 &=n r_+^{n-3}\frac{\phi(r_+)}{|u'(r_+)|}\frac{r_+^n}{P(r_+)}\\
 &>(n-2)r_-^{n-3}\frac{\phi(r_-)}{|u'(r_-)|}\frac{r_-^n}{P(r_-)}\\
 &=\frac{(n-2)r_-^{n-2}Q(r_-)}{P(r_-)u'(r_-)^2}.
 \end{aligned}
 \eqnum{eq4.13}
$$
Indeed, $r_+>r_-$ and $n\geq3$ imply
$r_+^{n-3}\geq r_-^{n-3}$.
Lemma~\ref{lem4.2} gives $\phi(r_+)>\phi(r_-)>0$, while strict
energy dissipation gives $|u'(r_+)|<|u'(r_-)|$.
Finally, Lemma~\ref{lem3.1} yields $\frac{r_+^n}{P(r_+)}>
\frac{r_-^n}{P(r_-)}$.
These comparisons justify the strict inequality
in~\eqref{eq4.13}. Combining these inequalities yields
 \begin{equation*}
 \frac{\Psi(r_+)}{|u'(r_+)|}
 >\frac{n\sigma r_+^{n-2}Q(r_+)}{P(r_+)u'(r_+)^2}
 >\frac{(n-2)\sigma r_-^{n-2}Q(r_-)}{P(r_-)u'(r_-)^2}
 >-\frac{\Psi(r_-)}{|u'(r_-)|}.
 \end{equation*}
Thus~\eqref{eq4.12} also holds in this case. Changing variables to the matched amplitude now gives
\begin{equation}{\label{eq4.14}}
\int_{t_-}^{r_\beta}H_*(u)\Psi\,dr
       +\int_{\bar r_\beta}^{t_+}H_*(u)\Psi\,dr
=\int_\beta^{a_*}H_*(\sigma)
       \biggl\{\frac{\Psi(r_+(\sigma))}{|u'(r_+(\sigma))|}
             +\frac{\Psi(r_-(\sigma))}{|u'(r_-(\sigma))|}\biggr\}\,d\sigma>0.
 \end{equation}

It remains to treat $[r_\beta,\bar r_\beta]$. If $\tau\le r_\beta$, $Q_n>0$ throughout this interval, and~\eqref{eq4.11} makes its integral positive. If $\tau>r_\beta$, Lemma~\ref{lem4.3} with $a=\eta$ makes the integral from $r_\beta$ to $\tau$ positive; after $\tau$, $Q_n>0$ and the remaining integral is positive as well. Hence the full integral in~\eqref{eq4.10} is positive. Since $T(t_-)>0$ by~\eqref{eq4.3}, $T(t_+)>0$. This is the only possible minimum and proves the lemma.
\end{proof}

\section{Radial nondegeneracy}\label{sec5}

We apply the comparison results of Section 5 to each nodal phase of an arbitrary bound state. This determines the zeros of its shooting variation and, together with its behaviour on the final tail, proves radial nondegeneracy. The argument does not assume uniqueness of the bound state.

\begin{theorem}\label{the5.1}
Let $u$ be a normalized radial $k$-node bound state, $k\ge1$, with $u(0)>0$. Its regular variation~\eqref{eq3.11} has exactly $k+1$ simple zeros on $(0,\infty)$ and satisfies $|v(r)|\to\infty$ as $r\to\infty$. Consequently,
$$
 \ker\bigl(-\Delta+1+\eps q|u|^{q-1}-p|u|^{p-1}\bigr)
             \cap\Hrad=\{0\}.
 \eqnum{eq5.1}
$$
\end{theorem}

\subsection{Starting from a nonzero critical point}
We show that the sign conditions obtained at a nonzero critical point persist through the next nodal phase. This allows the comparison results of Section 5 to be applied successively and also determines the sign of the variation after its last zero on the final tail.

\begin{proposition}\label{pro5.1}
Suppose $c>0$ is a critical point of $u$ and
$$
 Q(c)>0,\qquad M(c)>0,\qquad T(c)>0.
 \eqnum{eq5.2}
$$
If the next critical point is $d>c$, then the phase $[c,d]$ satisfies~\eqref{eq4.3} after simultaneous reflection if necessary, and
$$
 Q(d)>0,~~M(d)>0,~~T(d)>0.
 \eqnum{eq5.3}
$$
If $c$ is the last critical point, $v$ has exactly one further zero and, after the same reflection, is negative thereafter.
\end{proposition}

\begin{proof}
Reflect to arrange $u(c)>\beta$. Then $u'<0$ until the next critical point, or on the entire final tail. From the endpoint identities
$$Q(c)=c^nf(u(c))v(c),~~M(c)=-c^{n-1}u(c)v'(c),$$
we have $v(c)>0$ and $v'(c)<0$. Let $r_b$ be the subsequent radius at which $u=b$. If $v$ had no zero in $(c,r_b)$, then $\rho(c)>0$ and $\rho'>0$ up to $r_b$, while $\rho(r_b)\le0$ as in Lemma~\ref{lem4.1}. Thus its first zero $\tau$ lies in $(c,r_b)$, with $v'(\tau)<0$. On $(c,\tau)$, one has $u>b$ and $v>0$. Hence
$$(r^{n-1}v')'=-r^{n-1}f'(u)v<0,$$
so $v'<0$ throughout $[c,\tau]$. The identities in~(4.13)
give $Q>0$ and $M>0$ on $[c,\tau]$. Moreover,
$Q_1(c)=Q_2(c)=Q(c)>0$, while
$$Q_1'=r^{n-1}\{u'v'+f(u)v\}>0,~~Q_2'=2r^{n-1}u'v'>0$$
on $(c,\tau)$. Thus $Q_1,Q_2>0$ on $[c,\tau]$.

Since $u(\tau)>b$, the quantity $T$ is defined throughout
this interval. By~(4.13),
$$T'=g'(u)\left\{r^{n-1}h(u)v-u'M\right\}>0~~\text{on }[c,\tau].$$
Indeed, in the interior both terms in braces are positive;
at $c$ the first term is positive, and at $\tau$ the second
term is positive. Consequently, $T(r)\geq T(c)>0$ on
$[c,\tau]$. This argument uses $g'(u)>0$ and $h(u)>0$
for $u\geq b$, but does not require $g(u)\geq0$.

\emph{The incoming energy threshold.}
Let $r_\beta^-$ satisfy $u(r_\beta^-)=\beta$. We claim that $T'>0$ and $T>0$ on $[c,\max\{\tau,r_\beta^-\}]$. Only the case $\tau<r_\beta^-$ needs proof. On $[\tau,r_\beta^-]$, $g\ge0$, $g'>0$, $h>0$, and $\omega'>0$. At any zero of $v$,
$$
 T=r^{n-1}uv'\{g(u)-\omega\}.
 \eqnum{eq5.4}
$$
At $\tau$, $T>0$ and $v'<0$, hence $g(u(\tau))<\omega(\tau)$. Along the incoming arc $g(u)$ strictly decreases and $\omega$ strictly increases. A second zero $\sigma\le r_\beta^-$, with $v'(\sigma)>0$, would therefore give $T(\sigma)<0$. Since $T(\tau)>0$ and $T'(\tau)>0$, there would first be a point $t\in(\tau,\sigma)$ with $T(t)>0$, $T'(t)=0$ and $v(t)<0$. At that point~\eqref{eq3.16} gives
$$
 M(t)=\frac{t^{n-1}h(u(t))v(t)}{u'(t)}>0,
 \qquad Q(t)=P(t)\frac{v(t)}{u(t)}<0,
$$
so $T(t)=Q(t)-g(u(t))M(t)<0$, a contradiction. This excludes a second zero of $v$. Suppose that $t\in(\tau,r_{\beta^-}]$ is the first zero of $T'$.
Then $T(t)>T(\tau)>0$. Since $v(t)<0$, the identity $T'(t)=0$
in \eqref{eq3.13} gives
$$M(t)=\frac{t^{n-1}h(u(t))v(t)}{u'(t)}>0.$$
The first identity in \eqref{eq3.16} consequently yields
$$Q(t)=P(t)\frac{v(t)}{u(t)}<0.$$
As $g(u(t))\geq0$, we obtain
$T(t)=Q(t)-g(u(t))M(t)<0$, a contradiction.
Thus $T'>0$ and $T>0$ throughout $[\tau,r_{\beta^-}]$,
which proves the claim.

Put $R_0=\max\{\tau,r_\beta^-\}$. If $\tau<r_\beta^-$, $Q$ first increases and then decreases, ending at $Q(r_\beta^-)=T(r_\beta^-)>0$. If $\tau\ge r_\beta^-$, it increases throughout $[c,\tau]$. Thus $Q>0$ on $[c,R_0]$. The already proved signs of $Q_1,Q_2$ before $\tau$, and $u'v>0$ after it, imply $Q_1,Q_2>0$ there as well.

\emph{The low-amplitude passage.}
Suppose first that the next critical point $d$ exists. Let $z$ be the node and $r_\beta^+\in(z,d)$ the radius with $u=-\beta$. On $(r_\beta^-,r_\beta^+)$, one has $F(u)\leq0$,
with equality only at the node $z$. Since $u'<0$
throughout this interval, the regular bridge identity gives
$$B_0'=-2F(u)\frac{Q_n}{r(u')^2}.$$
At $R_0$ either $v=0$ or $F=0$, and therefore $B_0(R_0)=Q(R_0)>0$. If $Q_1>0$ or $v<0$ first failed before $r_\beta^+$, a zero $s$ of $Q_1$ would occur first: a return of $v$ to zero would have $v'>0$ and hence $Q_1=r^nu'v'<0$. Before $s$,
$$
 Q_n=Q_1+(n-1)r^{n-1}u'v>0,
$$
so $B_0(s)>B_0(R_0)>0$. But at $Q_1(s)=0$, equation~\eqref{eq4.2} gives $B_0(s)<0$. Thus $v<0$ and $Q_1>0$ persist to $r_\beta^+$, and $Q(r_\beta^+)=B_0(r_\beta^+)>0$.

On $[r_\beta^+,d]$, $Q'>0$ while $v<0$. A further zero of $v$ would give $Q<0$ in the interior, or $Q(d)=0$ at the endpoint, so none exists. This also gives $Q_1,Q_2>0$ on the remainder of the phase. Integrating $B_0'$ gives $Q(r_\beta^+)>Q(r_\beta^-)$ when $\tau\le r_\beta^-$. When $\tau>r_\beta^-$, first use $Q(\tau)>Q(r_\beta^-)$, then integrate from $\tau$ using $B_0(\tau)=Q(\tau)$. This proves the same strict threshold comparison. We have now verified all assumptions in~\eqref{eq4.3}.

\emph{The outgoing endpoint.}
Let $r_{-b}\in(z,d)$ satisfy $u=-b$. As in Lemma~\ref{lem4.1}, $\rho(r_{-b})>0$ and $\rho'>0$ up to $d$. Therefore
$$
 0<\rho(d)=-d^{n-1}f(u(d))v'(d),
 \qquad v'(d)>0,
$$
and $M(d)>0$. We already have $Q(d)>0$. All the inputs for Lemma~\ref{lem4.4} have been established without using its conclusion, so that lemma yields $T(d)>0$ and completes~\eqref{eq5.3}.

If $c$ is the last critical point, the incoming argument is unchanged. On the final tail, replace $r_\beta^+$ by infinity in the first-zero argument for $Q_1$. It proves $v<0$ after $\tau$ and $Q_1>0$ at all finite subsequent radii. Thus there is exactly one further zero of $v$.
\end{proof}

The endpoint recovery of $M$ uses $\rho$ only on $|u|\ge b$. No sign of $M$ is assumed on the low-amplitude passage; the negative lower-power contribution to $d(u)$ prevents such a monotonicity argument there.

\subsection{The final tail and the kernel}
It remains to control the variation on the final tail. We prove that its absolute value tends to infinity and then use this fact to exclude nonzero radial kernel elements, completing the proof of Theorem \ref{the5.1}.

\begin{lemma}\label{lem5.1}
Under the last-critical-point alternative in Proposition~\ref{pro5.1}, $|v(r)|\to\infty$ as $r\to\infty$.
\end{lemma}

\begin{proof}
After reflection, $u>0$, $u'<0$ on the final tail, and the preceding proposition gives $v<0$ after its last zero and $Q_1>0$. At $R_0=\max\{\tau,r_\beta^-\}$, $B_0(R_0)=Q(R_0)>0$. For $r>R_0$, $F(u)<0$ and $Q_n>0$, hence
$$
 B_0(r)\ge B_0(R_0)>0.
 \eqnum{eq5.5}
$$
For large $r$, $f'(u(r))<0$. The variation flux then satisfies
$$
 (r^{n-1}v')'=-r^{n-1}f'(u)v<0.
$$
It can change sign at most once, so $v$ is eventually monotone. If $|v|$ does not tend to infinity, eventual monotonicity
implies that $v$ is bounded. Fix a sufficiently large $R$.
Integrating the variation equation gives
$$r^{n-1}v'(r)=R^{n-1}v'(R)-\int_R^r t^{n-1}f'(u(t))v(t)\,dt.$$
Since $f'(u(t))$ and $v(t)$ are bounded on $[R,\infty)$,
the right-hand side is $O(r^n)$. Hence $v'(r)=O(r)$.

By Lemma~\ref{lem2.1}, $|u|+|u'|\le Ce^{-\kappa r}$ for some $\kappa>0$. Positive energy and $F(s)=-s^2/2+o(s^2)$ give
$$
 |u'|\ge\kappa_0|u|,
 \qquad \biggl|\frac{F(u)}{u'}\biggr|\le C|u|
 \quad\text{at large }r
$$
for a positive $\kappa_0$. If $v$ were bounded, every term defining $Q$ would consequently tend to zero, including $r^nu'v'$ because $v'=O(r)$. The other term in $B_0$ would also tend to zero. Thus $B_0(r)\to0$, contradicting~\eqref{eq5.5}.
\end{proof}

\begin{proof}[Proof of Theorem~\ref{the5.1}]
Lemma~\ref{lem4.1} and Lemma~\ref{lem4.4} give exactly one zero of $v$ in the first phase and $Q(c_1),M(c_1),T(c_1)>0$. Apply Proposition~\ref{pro5.1} successively at $c_1,\ldots,c_{k-1}$. Each finite phase adds exactly one simple zero and preserves the three positive endpoint quantities. Its last alternative applies at $c_k$, giving exactly one final zero. Lemma~\ref{lem5.1} gives $|v(r)|\to\infty$. Simplicity of every zero of $v$ follows from uniqueness for its linear equation at positive radii.

Let $w\in\Hrad$ be a weak kernel element. The potential is bounded and continuous. Local elliptic regularity, iterated from $H^1$ through $W^{2,t}$, makes $w$ locally bounded and then continuously differentiable; since the potential is locally H\"older continuous, $w$ is $C^2$. In particular its radial profile is regular at the origin, with $w'(0)=0$. The regular integral equation for the linearized equation gives $w=w(0)v$.

To see that $w$ decays, the radial $H^1$ estimate gives
$$
 |w(r)|\le C r^{(1-n)/2}\|w\|_{H^1(\R^n)}\quad(r\ge1),
$$
so $w(r)\to0$. This estimate follows, for smooth radial functions and then by approximation, by integrating $(r^{n-1}w^2)'$ from $r$ to infinity and applying the Cauchy--Schwarz inequality to the weighted $L^2$ norms of $w$ and $w'$. At large radii the coefficient in $-\Delta w+c(r)w=0$ satisfies $c(r)\ge1/2$. For any $0<\gamma<1/\sqrt2$, the radial function $Ce^{-\gamma(r-R)}$ is a positive supersolution there, since
$$
 (-\Delta+c(r))e^{-\gamma(r-R)}
 =\biggl(c(r)-\gamma^2+\frac{(n-1)\gamma}{r}\biggr)e^{-\gamma(r-R)}>0.
$$
Choose $C\ge|w(R)|$ and compare separately with $w$ and $-w$. The maximum principle on the exterior domain applies because all three functions tend to zero; a positive excess would attain its maximum at a finite interior point. Hence $w$ decays exponentially. This contradicts $w=w(0)v$ unless $w(0)=0$, and the regular initial value problem then gives $w\equiv0$.
\end{proof}

\section{Global continuation and uniqueness}\label{sec6}

We now combine radial nondegeneracy with the compactness theorem to prove global uniqueness. First, we obtain local solution branches and show that their nodal numbers are preserved. We then count all solutions with a fixed nodal number as the parameter varies.

\subsection{A compact Fredholm formulation}

To apply the implicit function theorem, we formulate the equation on the space of continuous radial functions vanishing at infinity. In this subsection, we show that its linearization is a Fredholm operator of index zero and use radial nondegeneracy to obtain its invertibility at every nodal bound state.

Let $X=\Xrad$, with the uniform norm, and let $R_0=(-\Delta+1)^{-1}$ be the Bessel-potential resolvent on $C_0(\R^n)$. Define
$$
 \mathcal T(\eps,u)
 =u-R_0\bigl(|u|^{p-1}u-\eps|u|^{q-1}u\bigr),
 \qquad (\eps,u)\in\R\times X.
 \eqnum{eq6.1}
$$
The map is $C^1$: the resolvent is bounded on $C_0$, and the scalar powers have uniformly continuous derivatives on bounded amplitude intervals. A bound state is a zero of $\mathcal T$. At a fixed state $U$,
$$
 D_u\mathcal T=I-R_0M_a,
 \qquad a=p|U|^{p-1}-\eps q|U|^{q-1}\in C_0(\R^n),
 \eqnum{eq6.2}
$$
where $M_a$ is multiplication by $a$.

\begin{lemma}\label{lem6.1}
The operator $R_0M_a:X\to X$ is compact. The operator $D_u\mathcal T$ is Fredholm of index zero, and its kernel agrees with the radial $H^1$ kernel of $-\Delta+1-a$. In particular it is invertible at every $k$-node state with $k\ge1$.
\end{lemma}

\begin{proof}
Approximate $a$ uniformly by continuous compactly supported radial functions $a_N$. For a bounded set of $w\in X$, the sources $a_Nw$ are uniformly bounded and have a common compact support. Elliptic interior estimates for $R_0(a_Nw)$ give equicontinuity on compact sets. The decay of the Bessel kernel gives uniform convergence to zero at infinity for this family. Arzel\`a--Ascoli on compact balls, followed by a diagonal argument and this common tail, proves compactness in $C_0$. Also
$$
 \|R_0M_{a_N}-R_0M_a\|_{X\to X}
 \le\|R_0\|_{C_0\to C_0}\|a_N-a\|_\infty\to0.
$$
Thus $R_0M_a$ is compact and $I-R_0M_a$ is Fredholm of index zero.

If $w\in X$ is in its kernel, then $-\Delta w+(1-a)w=0$. Since $1-a\ge1/2$ outside a ball, the exterior exponential barrier used in the proof of Theorem~\ref{the5.1} gives exponential decay of $w$. Uniform interior gradient estimates on unit balls at large radii give the same decay for $\nabla w$, so $w\in\Hrad$. Conversely, a radial $H^1$ kernel element is locally regular and belongs to $C_0$ by the radial estimate in that proof. The difference $w-R_0(aw)$ is an $H^1$ solution of $(-\Delta+1)\widetilde w=0$; testing against $\widetilde w$ makes it zero. Hence the kernels agree. Theorem~\ref{the5.1} gives triviality of the kernel, and the index-zero property then gives invertibility.
\end{proof}

\subsection{Preservation of the nodal number on local branches}

In this subsection, we prove that the local solution branches obtained by the implicit function theorem preserve the nodal number of the original bound state. We track its simple nodes on compact intervals and exclude additional nodes in the tail.

\begin{lemma}\label{lem6.2}
The implicit function theorem at a normalized $k$-node state $U$, $k\ge1$, gives a unique local $C^1$ branch in $X$. For nearby nonnegative parameters its members are bound states with the same nodal number and positive initial value.
\end{lemma}

\begin{proof}
Lemma~\ref{lem6.1} gives the local $C^1$ branch. Its $X$ convergence and local elliptic estimates give $C^1$ convergence of the radial profiles on compact intervals, including the origin. Thus $u(0)>0$ and each of the finitely many simple nodes of $U$ persists uniquely in a small neighbourhood.

Choose $R$ beyond the last node of $U$, with $U(R)\ne0$, and so large that the nearby branch is uniformly small on $[R,\infty)$. Its equation is
$$
 -\Delta u+c_{\eps,u}(r)u=0,
 \qquad c_{\eps,u}=1+\eps|u|^{q-1}-|u|^{p-1}\ge c_0>0
 \quad(r\ge R).
 \eqnum{eq6.3}
$$
Here $c_0$ can be chosen uniformly for nearby bounded parameters, even in a small real neighbourhood of $\eps=0$. Radiality makes the boundary value on $|x|=R$ a single nonzero value with the sign of $U(R)$. The exterior maximum principle shows that the entire tail has that sign; the strong maximum principle excludes an interior zero. Comparison with an exponential supersolution gives uniform exponential decay, and local gradient estimates give derivative decay. Hence the local solutions belong to $H^1$ and have no additional tail nodes. On the compact remainder, the preserved simple nodes and the nonvanishing values away from them exclude any additional zeros. Therefore their nodal number is exactly $k$.
\end{proof}

This argument does not assume that nodal number is an open condition on all of $H^1(\R^n)$. It establishes the condition for nearby solutions by controlling their tails.

\subsection{Counting all states}

We now count all bound states with a fixed nodal number, including any that might lie outside a chosen local branch. Compactness makes each solution set finite and excludes additional solutions appearing nearby, while local uniqueness shows that its cardinality is constant in the parameter. The pure-power result at the zero-parameter limit then gives uniqueness throughout the parameter range.

\begin{proof}[Proof of Theorem~\ref{thm:uniqueness}]
Fix $k\ge1$ and $E_*<\infty$. For each $\eps\in[0,E_*]$, let $\mathcal S_k(\eps)$ denote the normalized radial bound states with positive initial value and exactly $k$ nodes, and set
$$
 N_k(\eps)=\#\mathcal S_k(\eps).
 \eqnum{eq6.4}
$$
We first show that this is a finite nonnegative integer. Theorem~\ref{thm:compactness} makes each fibre compact in $X$, and Lemma~\ref{lem6.1}, together with the implicit function theorem, makes every
member isolated among the solutions at that parameter. An infinite compact fibre would have an accumulation point in the same fibre, contradicting local uniqueness. Thus $N_k(\eps)<\infty$.

If $\mathcal S_k(\eps_*)=\{U_1,\ldots,U_m\}$ is nonempty, choose disjoint implicit-function neighbourhoods of these states. For all sufficiently close parameters, each neighbourhood contains exactly one solution. Lemma~\ref{lem6.2} preserves its $k$ nodes. There can be no additional solutions outside these
neighbourhoods. Otherwise, there would exist parameters $\varepsilon_j\to\varepsilon_*$ and states $W_j\in S_k(\varepsilon_j)$ lying outside their union. By Theorem~\ref{thm:compactness}, after passage to a subsequence, $W_j\to W$ in $X$ for some $W\in S_k(\varepsilon_*)$.
Since $W$ is one of $U_1,\ldots,U_m$, the pairs $(\varepsilon_j,W_j)$ eventually belong to its implicit-function neighbourhood, a contradiction.
Hence $N_k=m$ on a relative neighbourhood of $\eps_*$.

The empty fibre must also persist locally. If $\mathcal S_k(\eps_*)=\varnothing$ but arbitrarily close coefficients had $k$-node states, the same compactness theorem would produce one at $\eps_*$, a contradiction. Therefore $N_k$ is locally constant on the entire connected interval $[0,E_*]$, including points where its value might be zero.

Tang's pure-power theorem~\cite[Theorem 1]{Tang} gives $N_k(0)=1$. It follows that $N_k(\eps)=1$ on $[0,E_*]$. Since $E_*$ is arbitrary, there is exactly one normalized state for every finite $\eps\ge0$. Its decay and strict central maximum follow from Lemma~\ref{lem2.1}, and radial nondegeneracy follows from Theorem~\ref{the5.1}.

Finally apply the scaling~\eqref{eq:normalization}. It preserves the centre, sign, nodes and decay, and the linearized potentials coincide because
$$
 \lambda q|u|^{q-1}=\eps q|U|^{q-1},
 \qquad \mu p|u|^{p-1}=p|U|^{p-1}.
$$
Thus all assertions for equation~\eqref{eq1.1} follow.
\end{proof}

\begin{proof}[Proof of Corollary~\ref{cor:global-branch}]
For each $\varepsilon\geq0$, Lemma~\ref{lem6.2} provides a local
$C^1$ solution map with values in $X$. The uniqueness
established in Theorem~\ref{thm:uniqueness} makes these local maps agree
on their overlaps. They therefore define a single global
$C^1$ map on $[0,\infty)$, with the derivative at zero
understood from the right.

To prove continuity in the remaining spaces, let
$\varepsilon_j\to\varepsilon_*\geq0$. The sequence of
parameters is contained in a compact interval. By
Theorem~\ref{thm:compactness}, every subsequence of
$u_{k,\varepsilon_j}$ has a further subsequence converging
in $H^1(\mathbb R^n)$ and in $C^1([0,A])$ for every
$A>0$ to a radial $k$-node bound state with coefficient
$\varepsilon_*$ and positive initial value.
Uniqueness identifies every such limit with
$u_{k,\varepsilon_*}$. Hence the full sequence converges
to $u_{k,\varepsilon_*}$ in these spaces.
The convergence in $C_0(\mathbb R^n)$ already follows
from the $C^1$ dependence in $X$.
\end{proof}

\begin{remark}\label{rem6.1}
The continuation argument needs estimates only on a compact coefficient interval and at one fixed nodal number. It neither requires a uniform bound as $\eps\to\infty$ nor assumes that all solutions lie on a known branch before their nondegeneracy is proved. In particular, it rules out additional disconnected components by counting the entire compact fibre.
\end{remark}

\appendix
\section{The planar comparison consequence}\label{appA}

The compactness argument and the bridge above were developed for the high-dimensional nondegeneracy theorem. In dimension two, an at-most-one statement for the same coefficient signs also follows directly from the general comparison theorem of Cort\'azar, Garc\'ia-Huidobro and Yarur~\cite[Theorem 1.1]{Cortazar}. We record the verification to specify the dimensional boundary of the argument.

\begin{proposition}\label{proA.1}
Let $1<q<p$, $\lambda,\mu>0$, and $k\ge1$. In $\R^2$, equation~\eqref{eq1.1} has at most one classical radial decaying solution with $u(0)>0$ and exactly $k$ nodes.
\end{proposition}

\begin{proof}
The cited theorem applies for $2\le n\le4$ to an odd continuous $f$ with $f(0)=0$, with $f\le0$ and not identically zero on $(0,b)$, $f>0$ on $(b,\infty)$, a zero $F(\beta)=0$ with $\beta>b$, and $F(s)\to\infty$. Its regularity assumptions are $f\in C^1(0,\infty)$ and $f'\in L^1(0,1)$. The additional condition is
$$
 \biggl(\frac{F}{f}\biggr)'(s)\ge\frac{n-2}{2}
 \quad(s>\beta).
 \eqnum{eqA.1}
$$
For $f(s)=-s-\lambda s^q+\mu s^p$ on $s>0$, extended oddly, the regularity and sign assumptions follow exactly as in Lemma~\ref{lem2.1}, since $f(s)/s^q=\mu s^{p-q}-s^{1-q}-\lambda$ is strictly increasing. It remains to verify~\eqref{eqA.1} for $n=2$.

For $s>b$, put $A=\mu s^{p-1}$, $B=\lambda s^{q-1}$, $D=A-B-1>0$. Then
$$
 K(s)=\frac{sf'(s)}{f(s)}
 =p+\frac{(p-q)B+(p-1)}{D}>p>1,
$$
and differentiation gives
$$
 K'(s)=\frac{(q-1)^2B-(p-1)^2A-(p-q)^2AB}{sD^2}<0.
 \eqnum{eqA.2}
$$
Here $A>1+B$ and $p-1>q-1>0$ make the numerator strictly negative. In particular $f'>0$. For $b<t<s$, monotonicity of $K$ implies
$$
 \log\frac{f(s)}{f(t)}
 =\int_t^s\frac{K(\xi)}{\xi}\dd\xi
 \ge K(s)\log\frac{s}{t},
$$
and hence $f(t)\le f(s)(t/s)^{K(s)}$. Since $F(b)<0$,
$$
 F(s)<f(s)\int_b^s(t/s)^{K(s)}\dd t
       <\frac{sf(s)}{K(s)+1}.
$$
For $s>\beta$, multiplication by $f'(s)/f(s)^2>0$ gives
$$
 \biggl(\frac{F}{f}\biggr)'(s)
 =1-\frac{F(s)f'(s)}{f(s)^2}
 >\frac{1}{K(s)+1}>0.
 \eqnum{eqA.3}
$$
This is~\eqref{eqA.1} in dimension two, so the cited theorem proves the proposition.
\end{proof}

\end{document}